\documentclass[a4paper,12pt]{amsart}          
\usepackage{amsfonts,amsmath,latexsym,amssymb} 
\usepackage{amsthm}      
\usepackage[dvipsnames]{xcolor}          
\usepackage{mathrsfs,upref}         
\usepackage{mathptmx}		    
\usepackage{float}
\usepackage{tikz}
\usepackage{circuitikz}
\usepackage{hyperref}
\usepackage{mathtools}
\usepackage{enumitem}
\allowdisplaybreaks

\newtheorem{theorem}{Theorem}          
\newtheorem{proposition}{Proposition}    
\newtheorem{lemma}{Lemma}               
\newtheorem{corollary}{Corollary}

\theoremstyle{definition}

\newtheorem{definition}{Definition}
\newtheorem{example}{Example}

\DeclareMathOperator\lspan{span}

\renewcommand{\P}{\mathcal{P}}
\newcommand{\K}{\mathbb{K}}
\newcommand{\R}{\mathcal{R}}
\renewcommand{\L}{\mathcal{L}}
\newcommand{\weight}{b}

\author{Anna Muranova}
 \address{Anna Muranova: Faculty of Mathematics and Computer Science,
University of Warmia and Mazury in Olsztyn,
ul. Sloneczna 54, 10-710 Olsztyn, Poland} 

 \email{anna.muranova@matman.uwm.edu.pl}


    \title[Laplacian on graphs]{Discrete Laplace and transition operators over non-Archimedean ordered fields}

\begin{document}

\maketitle

\medskip
    \begin{abstract}

We investigate properties of spectrum of normalized Laplacian $\L$ for finite graphs over non-Archimedean ordered fields. We prove a Cheeger's inequality for first non-zero eigenvalue. Then we describe properties of the operator $\P=I-\L$, which is a generalization of transition operator. We show that Cheeger estimate $\alpha_1\preceq \sqrt{1-h^2}$ for the second largest eigenvalue of $\P$ is crucial for investigation of the convergence of analogue of random walk to equilibrium over a non-Archimedean ordered fields. We consider examples over the Levi-Civita field.
    \end{abstract}

\section{Introduction}

The eigenvalues of normalized Laplacian $\mathcal L$ on weighted graphs is a classical widely investigated topic.
For finite graphs the results are presented in the classical books and papers on graph theory, like e.g. \cite{BauerJost}, \cite{Chung}, \cite{Grigoryan}. The probability (transition) operator $\P$ is the difference of identity operator and Laplacian. It uniquely describes a random walk on graph. The most important results in classical case are that all eigenvalues of Laplacian lie in $[0,2]$, that $0$ is always an eigenvalue and that the eigenvalues are symmetric with respect to $1$ if and only if the graph is bipartite. The branch of results for probability operator follows from the above mentioned facts. Further one can consider the Hilbert space of functions from vertices of graph to $\mathbb R$ and investigate properties of $\P$ and $\L$ as operators on this space. It is known, that random walk convergence to an equilibrium for any given function on finite graph (see e.g. \cite{Chung}, \cite{Grigoryan}).

We investigate the Laplace and transition operators on graphs, whose weighted belongs to a non-Archimedean real-closed ordered field $\K$, following the previous works by author (\cite{Muranova1}, \cite{Muranova2}). We show, that the results, arising from linear algebra hold also in this settings. More precisely, we present the prove, that eigenvalues of Laplacian belong to the same field, are between $0$ and $2$. Further, they are symmetric with respect to $1$ if and only if the graph is bipartite.
Moreover, some inequalities in terms of number of vertices also hold for them (see e.g. Proposition \ref{Prop1}). We present the direct  proofs of all these facts.

Since we do not use non-standard probability theory (e.g. \cite{Nelson}), to avoid a confusion we call the operator $\P$  \emph{transition operator} in this paper, keeping in mind, that its matrix has values between $0$ and $1$ in $\K$. The details about operators are described in Section \ref{section::LP}.

In non-Archimedean fields the fact that eigenvalues of $\P$ are between $-1$ and $1$ is not enough, to show the convergence of $\P^m f, m\in \Bbb N$, $f$ is a function on vertices, taking values in $\Bbb K$ to an equilibrium (function on vertices, depended on $f$). Moreover, the convergence it not always the case. One of the possible ways to investigate the convergence is to use Cheeger inequality (see Theorem \ref{Theorem::Cheeger}):
$$
\lambda_1\succeq 1-\sqrt{1-h^2}
$$
for the smallest non-zero eigenvalue of Laplacian, which leads to $\alpha_1\preceq \sqrt{1-h^2}$ for the second smallest eigenvalue of $\P$. Further we show, that in case of non-Archimidean ordered field there is a crucial difference between two classical formulations of Cheeger inequality: 
$$\lambda_1\succeq {\frac{h^2}{2}}
\;\;\;\;\;\;\mbox{ and }\;\;\;\;\;\;
\lambda_1\succeq 1-\sqrt{1-h^2}
$$
since only the latest one can be used to prove the convergence. 

In Section \ref{section::rw} we investigate the convergence of $\P^m f$ or $\P^{2m} f$ for some $f$ to equilibrium, using inequalities, which in the classical (real) case are used to calculate the rate of convergence. In Theorem~\ref{Prop3} we show that, on the contrary to real case, for any non-bipartite non-complete graph over non-Archimedean field there is always a subspace of functions, for which  $\P^m f$ does not converge. Complete graphs show different behavior due to Lemma \ref{lemma::non-complete} and we do not describe them in this note.

The largest non-Archimedean ordered field is the field of surreal numbers (\cite{Conway}, \cite[Theorem 24.29]{BB}), and it is known (e.g. \cite{RSS}) that the convergent sequences in classical sense are  exclusively constant sequences there. Therefore, it makes sense to investigate convergence only in some particular non-Archimedean ordered fields. Following out previous work (\cite{Muranova2}), which shows that the Levi-Civita field appears naturally in theory of electrical networks, in the last section we investigate the convergence of $\P^m$ over the Levi-Civita field. Theorem \ref{corhLC} gives the sufficient condition for convergence for any function on bipartite graph in terms of Cheeger constant. For non-bipartite non-complete graphs the Theorem \ref{corhLC1} describes sufficient conditions for existence of function $f$ on vertices, such that $\P^m f$ converges to an equilibrium.

\section{Preliminaries}
\label{Section::preliminaries}
Let $\K$ be an arbitrary real-closed non-Archimedean ordered field (in fact, an order can be introduced in any real-closed field \cite[Theorem 2.2]{Lang}).  If the ordered field is not real closed, we should consider one of its real closures, which are all isomorphic due to Artin-Schreier Theorem (see e.g. \cite{DalesHughWoodin}) with isomorphism, preserving the order. A field is called \emph{ordered} if the property of positivness ($\succ 0$) is defined for its elements, and if it satisfies the folowing postulates:
\begin{enumerate}
\item
For every element $a\in \K$, just one of the relations 
$$
a\succ 0, a=0, -a\succ 0
$$
is valid.
\item
If $a\succ 0$ and $b\succ 0$, then $a+b\succ 0$ and $ab\succ 0$. We write $a\succ b$ (``$a$ greater than $b$'') if $a-b\succ 0$. We write $a\prec b$ if $b\succ a$ (``$a$ smaller than $b$''). This is a total order in $\K$. The signs $\preceq$ and $\succeq$ correspond to ``great or equal'' and ``smaller or equal'' 
(see \cite{Waerden}).
\end{enumerate}
The set $\K^+=\{a \in \K\mid a\succ 0\}$ is called a \emph{set of positive elements}.

An ordered field is  \emph{non-Archimedean}, if there exist an \emph{infinitesimal}~$\tau\in \K^+$, i.e, $\tau$ such that $\tau\preceq 1/m$ for any $m\in \Bbb N$, where $m=\underbrace{1+\cdots+1}_{m}\in \K^+$.

An ordered field $\K$ is \emph{real-closed} if it has no proper algebraic extension to an ordered field, or, equivalently, if the complexification of $\K$ is algebraically closed, or, again equivalently, if every positive element in $\K$ has a square root and every polynomial over $\K$ of odd degree has a root in $\K$ \cite[p. 38]{DalesHughWoodin}. 

Let $\K$ be a real-closed field. Further we need a complexification of $\K$, i.e a field extension $\K(\sqrt{-1})$, which is algebraically closed. Let us agree, that we pick up one square root of  $-1$ and denote it by $\mathbf i$. Then we can consider the automorphism of this field, which maps elements of $\K$ into themselves, and maps $\mathbf i$ to $-\mathbf i$. We call this automorphism a \emph{complex conjugation} and we denote it further in text by overlining the element (in the same way as case of usual complex conjugation). We denote $\K(\mathbf i)$ by $\K_C$.
\begin{definition}\cite{Muranova1}
Let $\K$ be an ordered field. A \emph{graph (over an oredred field $\K$)} is a couple $(V, b)$, where $V$ is a set of vertices (i.e. arbitrary set)
and $\weight:V\times V\rightarrow \K$ satisfies the following properties:
\begin{enumerate}
\item
$\weight (x,y)\succeq 0$ for any $x,y\in V$,
\item
$\weight (x,y)=\weight (y,x)$ for any $x,y\in V$,
\item
$\weight (x,x)= 0$ for any $x\in V$.
\end{enumerate}
If $\weight (x,y)\ne 0$, we say that there is an \emph{edge} between $x,y$ and write $x\sim y$. 
\end{definition}
The third condition means, that we consider graphs without loops.
When the field $\K$ is relevant, we will say \emph{graph over the field $\K$}.
Note that if instead of $\K$ we use $\mathbb R$, then we obtain a classical definition of a \emph{weighted graph}.

The \emph{path} in graph is any sequence of vertices $x_1,x_2,\dots,x_n$ such that 
$$
x_1\sim x_2\sim\dots \sim x_n.
$$
A graph is called \emph{connected} if there exist a path between any two of its vertices. A graph is called \emph{finite} if its set of vertices is finite ($\#V<\infty$).
In this note we consider exclusively finite connected graphs.

A graph is called \emph{complete} if $x\sim y$ for any $x,y\in V$. Otherwise it is called \emph{non-complete}.

A graph is called \emph{bipartite}, if there exist a partition of its vertices $V=~V_1\cup ~V_2$ such that if $x\sim y$ for $x,y\in V$ then either $x\in V_1, y\in V_2$ or $x\in V_2, y\in V_1$. Otherwise, the graph is called \emph{non-bipartite}.

Let us consider the \emph{normalized} weights 
\begin{equation}
p(x,y)=\dfrac{\weight (x,y)}{\weight (x)},
\end{equation}
where $\weight (x)=\sum_y \weight (x,y)$. Note, that $p(x,y)\ne p(y,x)$ in general.

It is known that weighted graph can be considered as Hilbert space (e. g. \cite{Grigoryan}, \cite{KellerBook}). We would like to introduce an analogue of scalar product on graphs over ordered fields.

Let us consider a set of functions on vertices of a graph $(V,b)$:
\begin{equation*}
\mathfrak F=\{f\mid f:V\rightarrow \K\}
\end{equation*}

As an analogue of \emph{scalar} product for any $f,g\in \mathfrak F$ we consider
\begin{equation*}
\langle f,g\rangle=\sum_{x\in V} f(x)g(x)\weight (x).
\end{equation*}

We list below the properties  of this scalar product, which are crucial for us. These properties coincide with the properties of usual scalar product for order $>$ on $\mathbb R$.
\begin{enumerate}
\item 
for any $f,g\in \mathfrak F$ holds
\begin{equation*}
\langle f,g\rangle=\langle g,f\rangle;
\end{equation*}
\item 
for any $a,b\in \K$ and for any $f,g\in \mathfrak F$ holds
\begin{equation*}
\langle af+bf,g\rangle=a\langle f,g\rangle+b\langle f,g\rangle;
\end{equation*}
\item
for any $f\in \mathfrak F$ 
\begin{equation*}
\langle f,f\rangle\succeq 0
\end{equation*}
and  $\langle f,f\rangle = 0$ if and only if $f\equiv 0$ on $V$.
\end{enumerate}
Moreover, the analogue of complex-valued scalar product can be considered for the functions, taking values in $\K_C$. We denote the set of functions on vertices with values from $\K_C$ by 
\begin{equation*}
\mathfrak F_C=\{f\mid f:V\rightarrow \K_C\}.
\end{equation*}
Then the `complex' scalar product for $f,g\in \mathfrak F_C$ is defined by
\begin{equation*}
\langle f,g\rangle_C=\sum_{x\in V} f(x)\overline {g(x)}\weight (x).
\end{equation*}
It posses properties (2) -- (3) and is anti-symmetric, i.e. 
\begin{enumerate}[label=(1*)]
\item for any $f,g\in \mathfrak F_C$ $$\langle f,g\rangle_C=\overline{\langle g,f\rangle}_C.$$
\end{enumerate}

Obviously, $\langle f,g\rangle_C=\langle f,g\rangle$ for any $f,g\in \mathfrak F$.

These things are usually described in frameworks  of symmetric and Hermitian quadratic forms (see e.g. \cite{Lang}), but we would like to bring here together several things and introduce an  analogue of a pre-Hilbert space on finite graphs over ordered fields.

We define 
$$\|f\|=\sqrt {\langle f,f\rangle}\in \K^+
$$ 
for any $f\in  \mathfrak F_c$.

\section{Laplace and transition operators, their eigenvalues}\label{section::LP}

Our main point of investigation will be a normalized Laplace operator (Laplacian) and transition operator on $\mathfrak F$.  
\emph{Laplacian} on a graph $(V,b)$ over $\K$ is defined as $\L:\mathfrak F\to \mathfrak F$:
\begin{equation}
\mathcal L f(x)=\sum_{y\in V} (f(x)-f(y))\dfrac{\weight (x,y)}{\weight (x)}=\sum_{y\in V} (f(x)-f(y))p(x,y)
\end{equation}
for any $f \in \mathfrak F$.

Further, we define a \emph{transition operator} as  $\mathcal P=I-\mathcal L$ on $f \in \mathfrak F$. The entries of the corresponding matrix over $\K$ lie between $0$ and $1$, moreover, elements of each row/column sum up to $1\in\K$.

Note that for the classical weighted graphs $\mathcal P$ is a stochastic matrix of the corresponding Markov chain (see e.g. \cite{LPW}, \cite{Woess09} just to mention some among many). 

The Laplacian is a symmetric operator with respect to the introduced scalar product (i.e. $\langle\mathcal L f,g\rangle=\langle f,\mathcal L g\rangle$ for $f,g \in \mathfrak F$), due to the \emph{Green formula} \cite[Thm. 22]{Muranova1}:
\begin{equation*}
\sum_{x\in V}\mathcal L f(x)g(x)\weight (x)=\dfrac{1}{2}\sum_{x,y\in V}(\nabla_{xy}f)(\nabla_{xy}g)\weight (x,y),
\end{equation*}
where $\nabla_{xy}f=f(y)-f(x)$ for any $f,g \in \mathfrak F$.

Moreover, it is symmetric for the scalar product $\langle \cdot , \cdot\rangle_C$ and  all the functions from $\mathfrak F_C$, since the analogue of Green formula for function from $\mathfrak F_c$ can be proven in the classical way (see e.g proof in \cite{Grigoryan} or proof in \cite[Thm. 22]{Muranova1}). Therefore,
\begin{align*}
&\sum_{x\in V}\mathcal L f(x)\overline{g(x)}\weight (x)=\sum_{x\in V}\mathcal L f(x)\overline{g(x)}\weight (x).
\end{align*}
Now we present here several lemmas, which are reformulations of known algebraic facts for the Laplacian as a linear operator on $\mathfrak F$. Moreover, we present proofs for some of them, since not all linear algebra techniques can be used over an arbitrary non-Archimedean ordered field.
\begin{lemma}
All eigenvalues of the Laplacian belong to $\K$. Moreover, all its eigenfunctions  belong to $ \mathfrak F$.
\end{lemma}
\begin{proof}
Let $\lambda\in \K_C$ be an eigenvalue of $\mathcal L$ with an eigenfunction $f$, i.e. there exists $f\in \mathfrak F_c$, $f\not\equiv 0$ such that $\mathcal L f(x)=\lambda f(x)$ for any $x\in V$. Then 
\begin{align*}
&\lambda\langle f,f\rangle_C=\langle \lambda f,f\rangle_C=\langle \mathcal L f,f\rangle_C=\langle  f,\mathcal L f\rangle_C=\langle f,\lambda  f\rangle_C=\overline{\lambda}\langle f,f\rangle_C\\
\end{align*}
due to anti-symmetry. Therefore, dividing by $\langle f,f\rangle_C$ we obtain $\overline{\lambda}=\lambda$ from where follows, that $\lambda \in \K$.

Let $\lambda$ has an eigenfunction  $f\in \mathfrak F_C$. Then, there exists $f_1, f_2\in \mathfrak F$ such that $f=f_1+\mathbf i f_2$, and due to the linearity of $\mathcal L$,
\begin{equation*}
\mathcal L f =\mathcal L f_1+\mathbf i\mathcal L f_2=\lambda f_1+\mathbf i\lambda f_2
\end{equation*}
from where follows
\begin{align*}
\mathcal L f_1&=\lambda f_1\\
\mathcal L f_2&=\lambda f_2.
\end{align*}
Moreover, since $f\not\equiv 0$, at least one of the functions $f_1, f_2$ is not constantly zero and, therefore, it is an eigenfunction of $\mathcal L$.

\end{proof}

\begin{lemma}
All the eigenvalues of the Laplacian are precisely the roots of characteristic polynomial of its matrix.
\end{lemma}
\begin{proof}
Follows from more general theorem, see e.g. \cite[Theorem 3.2 on p.562]{Lang}
\end{proof}

\begin{lemma}
There exists an orthonormal basis ($\langle v_i, v_j\rangle = 0$ and $\|v_i\|=1$ for any $i,j=\overline{1,n}, i\ne j$) of eigenfunctions of the Laplacian, i.e. the matrix of Laplacian is diagonilazable over $\K$.
\end{lemma}
\begin{proof}
Follows from the proof of Lemma 6.3 in \cite[p. 582]{Lang}, and proof of the Spectral Theorem 6.4 there.
\end{proof}

Therefore, the Laplacian has exactly $n=\#V$ eigenvalues 
$$
\lambda_0\preceq\lambda_1\preceq\dots \preceq\lambda_{n-1}
$$ 
and there are $n$ orthonormal eigenfunctions $v_0,v_1\dots,v_{n-1}$, forming a basis of $\mathfrak F$. These notations for eigenvalues and eigenfunctions of the Laplacian will be used through the entire paper.

\begin{lemma}\label{Rq}
Set 
\begin{equation*}
\mathcal R(f)=\dfrac{\langle\mathcal L f,f\rangle}{\langle f,f\rangle}
\end{equation*}
for $f\in \mathfrak F\setminus\{0\}$ (the function $\mathcal R$ is called the \emph{Rayleigh quotient} of $\mathcal L$). The following is true for all $k=0,\dots,n-1$:
\begin{equation}\label{lambdaMax}
\lambda_k\preceq \mathcal R(f) \mbox{ for any } f\in\mathfrak F\mbox{ with }\langle f,v_i\rangle=0\; \mbox{ for all }\; i=\overline{0,k-1},
\end{equation}
 \begin{equation}\label{lambdaMin}
\lambda_k \succeq \mathcal R(f) \mbox{ for any }{f \in\mathfrak F\mbox{ with }\langle f,v_i\rangle=0\; \mbox{ for all }\; i=\overline{k+1,n-1}}.
\end{equation}
Moreover,
 \begin{equation}
\lambda_k =\min_{\substack{f:\langle f,v_i\rangle=0\\ \forall\; i=0,\dots,k-1}}\mathcal R(f) =\max_{\substack{f:\langle f,v_i\rangle=0\\ \forall\; i=k+1,\dots,n-1}}\mathcal R(f),
\end{equation}
where maximum and minimum are attained on the corresponding eigenfunction $v_k$.
\end{lemma}
\begin{proof}
The fact that $\lambda_k=\mathcal R(v_k)$ follows from the definitions of Rayleigh quotient and an eigenvalue.

The inequalities \eqref{lambdaMax} and \eqref{lambdaMin} follow from considering
Rayleigh quotient for the basis decomposition  $f=f_{k}v_{k}+\dots+f_{n-1}v_{n-1}$ and $f=f_0v_0+\dots+f_{k}v_{k}$ for any function from the given subspaces.

\end{proof}

Further we present some facts, which are known for classical weighted graphs (see e.g. \cite{Chung}, \cite{Grigoryan}) and have their analogues on graphs over ordered fields. We also comment on their proofs for graphs over ordered fields.
\begin{theorem}\label{thm1}
For any finite graph over with $\#V>1$ the following is true:
\begin{enumerate}
\item
$\lambda_0=0$ is a simple eigenvalue of $\mathcal L$ with constant eigenfunction.
\item
An eigenfunction $v_1$, corresponding to the first non-zero eigenvalue satisfies $\langle v_1,1\rangle=0$
\item
All the eigenvalues of $\mathcal L$ are contained in $[0,2]$.
\item
If the graph is not bipartite then all the eigenvalues of $\mathcal L$ are in $[0,2)$.
\item
If the graph is bipartite and $\lambda$ is an eigenvalue of $\mathcal L$, then $2-\lambda$ is an eigenvalue of $\mathcal L$ with the same multiplicity.  
\end{enumerate}
\end{theorem}
\begin{proof}
Follows exactly the same outline as the proof of Theorem 2.3 and Theorem 2.6 in \cite{Grigoryan}.
\end{proof}

\begin{corollary}[see e.g. \cite{Chung}, \cite{Grigoryan}]
$\lambda_{n-1}=2$ is a simple eigenvalue of any bipartite graph with eigenfunction $c=const$ on $V_1$ and $-c$ on $V_2$ (where $V=V_1\cup V_2$ is the partition).
\end{corollary}

\begin{proposition}\label{Prop1}
For any finite graph with $\#V=n$ the following is true:
\begin{equation}\label{sum:of:ev}
\sum_{i=1}^{n-1} \lambda_i = n.
\end{equation}
Consequently,
\begin{equation}\label{lambda1:n:n-1}
\lambda_1\preceq\dfrac{n}{n-1}
\end{equation}
and
\begin{equation}\label{lambdan-1:n:n-1}
\lambda_{n-1}\succeq\dfrac{n}{n-1}
\end{equation}
\end{proposition}
\begin{proof}
Equality \eqref{sum:of:ev} follows from the fact, that the trace of a matrix is independent of the basis. The latest can be proven over arbitrary ordered field by considering free coefficient of the trace of the characteristic polynomial of the matrix.
The equations \eqref{lambda1:n:n-1} and \eqref{lambdan-1:n:n-1} follows from \eqref{sum:of:ev} using $\lambda_ 0=0$.
\end{proof}

Now let us consider the operator $\mathcal P=I-\mathcal L$, where $I$ is an identity matrix $n\times n$.  Due to Theorem \ref{thm1} its  eigenvalues are
$$
\boxed{1=\alpha_0\succ\alpha_1\succeq\alpha_2\dots\succeq \alpha_{n-1}\succeq-1}
$$
where $\alpha_i=1-\lambda_i$) and belong to $\Bbb K$.  Moreover, we have
\begin{proposition}
For any finite graph with $\#V>1$ the following is true:
\begin{enumerate}
\item
The eigenfunction of  $\mathcal P$, corresponding to the eigenvalue $1$, is constant.
\item
If the graph is not bipartite then all the eigenvalues of $\mathcal P$ are in $(-1,1]$.
\end{enumerate}
\end{proposition}
Again, we will use these notations for eigenvalues of $\P$ through the entire note. Note, that its eigenfunctions are the same as for the corresponding Laplacian.

\section{Estimates on eigenvalues}
In this Section we present estimates on eigenvalues of Laplace and transition operaror for finite graph $(V,b)$ over non-Archimedean ordered field. The results are also known to hold for the field $\Bbb R$ (see e.g. \cite{BauerJost}, \cite{Grigoryan}).
In the classical case of weighted graphs the absolute values of eigenvalues $\alpha_1$ and $\alpha_{n-1}$ (the latest plays role in the case of non-bipartite graphs) determine the rate of convergence of random walk, i.e. the speed of convergence of $\P ^m f, m\in \Bbb N$ (or $\P ^{2m} f$ for bipartite graphs) for $f\in\mathfrak F$ to a function $\overline f\in \mathfrak F$, depending on $f$ and is related to a mixing time (see e.g. \cite{Grigoryan}, \cite{LPW}).
In this Section, we present estimates on these eigenvalues for graphs over ordered fields. The main result here is the Cheeger's inequality in its stronger formulation for graphs over ordered fields (Theorem \ref{Theorem::Cheeger}).

By \emph{absolute value of $k\in \K$} we mean the value $|k|\in \K^+\cup\{0\}$ defined as follows
\begin{equation*}
|k|=\begin{cases}
k, \mbox{ if }k\succeq 0\\
-k, \mbox{otherwise.}
\end{cases}
\end{equation*}

\subsection{Cheeger's inequality}
\begin{lemma}\label{lemma::non-complete}
For any non-complete graph $\alpha_1\succeq 0$.
\end{lemma}
\begin{proof}
The proof follows the classical outline for the proof of fact, that $\lambda_1\le 1$, see e.g. \cite[p.44]{Grigoryan}. We present the proof here in a shorter form.

Let us take two vertices $z_1,z_2$ such that $z_1\not\sim z_2$ (it is possible, since the graph is non-complete) and consider  function
\begin{equation*}
f(x) = 
\begin{cases}
c_1, \mbox{ for } x = z_1\\
c_2, \mbox{ for } x = z_2\\
0, \mbox{ otherwise},
\end{cases}
\end{equation*}
where $c_1, c_2\in \K$, such that $\langle f,1\rangle=0$. Let us denote $V_{12}=V\setminus\{z_1,z_2\}$. Then $f(x)=0$ for $x\in V_{12}$.

We have $\langle f, f\rangle=c_1^2\weight (z_1)+c_2^2\weight(z_2)$ and, due to the Green formula,
\begin{align*}
\langle \mathcal L f,f\rangle&=\dfrac{1}{2}\sum_{x,y\in V}(f(y)-f(x))^2 \weight (x,y)\\
&=\dfrac{1}{2}\left(\sum_{\substack{x=z_1,\\ y\in V_{12}}}(f(y)-f(x))^2 \weight (x,y)+\sum_{\substack{x\in V_{12},\\y=z_1}}(f(y)-f(x))^2 \weight (x,y)\right)\\
&+\dfrac{1}{2}\left(\sum_{\substack{x=z_2, \\y\in V_{12}}}(f(y)-f(x))^2 \weight (x,y)+\sum_{\substack{x\in V_{12},\\y=z_2}}(f(y)-f(x))^2 \weight (x,y)\right)\\
&+\underbrace{\dfrac{1}{2}\left(\sum_{\substack{x=z_1,\\y=z_2}}(f(y)-f(x))^2 \weight (x,y)+\sum_{\substack{x=z_2,\\y=z_1}}(f(y)-f(x))^2 \weight (x,y)\right)}_{=0 \mbox{ since }z_1\not\sim z_2}\\
&=\sum_{\substack{x=z_1,\\y\in V_{12}}}(f(y)-f(x))^2 \weight (x,y)+\sum_{\substack{x=z_2,\\y\in V_{12}}}(f(y)-f(x))^2 \weight (x,y)+0\\
&=c_1^2\weight (z_1)+c_2^2\weight (z_2).
\end{align*}
Therefore, $\R(f)=1$. Then  from Lemma \ref{Rq} follows, that $\lambda_1\preceq 1$ and $\alpha_1\succeq~0$.
\end{proof}
Further in this subsection we discuss Cheeger constant for graphs over ordered fields. Cheeger isoperimetric constant was firstly introduced by Jeff Cheeger for manifolds in \cite{Cheeger} and then generalized to graphs (see e.g. \cite{BauerJost}, \cite{Chung}).

Firstly, we introduce the following notations for $\emptyset\ne S\subset V$:
\begin{equation*}
\weight(S)=\sum_{x\in S}\weight (x)
\end{equation*}
and
\begin{equation*}
\weight(\partial S)=\sum_{x\in S,y\in V\setminus S}\weight (x,y),
\end{equation*}
i.e. $\partial S$ is a subset of edges of $V$, which have their ends in both $S$ and $V\setminus S$.
\begin{definition}
Let $V$ be a finite connected graph with at least two vertices. The \emph{Cheeger constant} is defined by
\begin{equation}\label{Def:Cheeger}
h = \min_{\emptyset\ne S\subset V, \weight(S)\preceq\frac{1}{2}\weight(V) }\dfrac{\weight(\partial S)}{\weight(S)}= \min_{\emptyset\ne S\subset V}\dfrac{\weight(\partial S)}{\min\{\weight(S),\weight(V\setminus S)\}}
\end{equation}
\end{definition}
It holds $h\preceq 1$ (consider $S$, consisting of the one vertex with the smallest weight).

The most known estimates of the first non-zero eigenvalue of Laplacian for classical weighted graphs in terms of Cheeger constants are 
\begin{equation*}
\lambda_1\ge 1-\sqrt{1-h^2}
\end{equation*}
and
\begin{equation*}
\lambda_1\ge \dfrac{h^2}{2},
\end{equation*}
where the second follows from the first, but the second is easier to prove (\cite{BauerJost}, \cite{Chung}).
These estimates obviously lead to the following estimates for the eigenvalue $\alpha_1$ of transition operator for non-complete graphs: $\alpha_1\le \sqrt{1-h^2}$  and $\alpha_1\le 1-\frac{h^2}{2}$. 
In a non-Archimedean ordered field for the r.h.s. of the latest inequality holds 
$$
1-\frac{h^2}{2}\succeq \frac{1}{2}\succ \tau
$$
for any infinitesimal $\tau$. The latest means, that one can not use weaker version of Cheeger's inequality to investigate convergence of the analogue of random walk in non-Archimedean case (for the details see Section 5). Therefore, our next step is to prove the stronger Cheeger's inequality for graphs over ordered fields. Our proof is similar to the classical proof by Chung \cite{Chung} for unweighted graphs and to proof in \cite{BauerJost}, but the latest proof use integrals, which we can not use over ordered fields. We also do not use integral notations as, for example, in \cite{Grigoryan}. Moreover, we write our proof very accurately, to be sure that we use exclusively properties of an ordered field.

\begin{theorem}[Cheeger's inequality]\label{Theorem::Cheeger}
\begin{equation*}
\lambda_1\succeq 1-\sqrt{1-h^2}
\end{equation*}
\end{theorem}

Firstly, we prove several Lemmas for $\lambda_1$.

\begin{lemma}\label{Lemma:lambda1g}
\begin{equation}\label{lambda1g}
\lambda_1\succeq \dfrac{\frac{1}{2}\sum_{x,y}\weight (x,y)(g(x)-g(y))^2}{\sum_x \weight (x)g(x)^2},
\end{equation}
where
\begin{equation*}
g(x)=
\begin{cases}
v_1(x), \mbox{ for }v_1(x)\succ 0\; (\Leftrightarrow x\in V^+),\\
0, \mbox{ otherwise } (\Leftrightarrow x\in V^-).
\end{cases}
\end{equation*}
Note that $V^-$ is non-empty, since $\langle v_1, 1\rangle=0$.
\end{lemma}

\begin{proof}
From the one hand we have
\begin{equation*}
\langle\L v_1,g\rangle=\sum_x \weight (x) \L v_1(x)g(x)=\lambda_1\sum_x \weight (x)  v_1(x)g(x)=\lambda_1\sum_xg^2(x)\weight (x).
\end{equation*}
From the other hand, due to the Green's formula
\begin{align*}
\langle\L v_1,g\rangle&=\dfrac{1}{2}\sum_{x,y} \weight (x,y) (v_1(y)-v_1(x))(g(y)-g(x))\\
&=\dfrac{1}{2}\sum_{x,y\in V^+} \weight (x,y) (g(y)-g(x))^2\\
&+\underbrace{\dfrac{1}{2}\sum_{x,y\in V^-} \weight (x,y)(v_1(y)-v_1(x)) \underbrace{(g(y)-g(x))}_{=0\mbox{ on }V^-}}_{=0=\frac{1}{2}\sum_{x,y\in V^-} \weight (x,y) (g(y)-g(x))^2}\\
&+\underbrace{\dfrac{1}{2}\sum_{\substack{x\in V^-,\\y\in V^+}} \weight (x,y)\underbrace{(v_1(y)-v_1(x))}_{\succeq g(y)-g(x)} \underbrace{(g(y)-g(x))}_{=g(y)\succeq 0}}_{\succeq (g(y)-g(x))^2}\\
&+\underbrace{\dfrac{1}{2}\sum_{\substack{x\in V^+,\\y\in V^-}} \weight (x,y)\underbrace{(v_1(y)-v_1(x))}_{\preceq g(y)-g(x)} \underbrace{(g(y)-g(x))}_{=g(y)\preceq 0}}_{\succeq (g(y)-g(x))^2}\\
&\succeq\dfrac{1}{2}\sum_{x,y} \weight (x,y) (g(y)-g(x))^2\\
\end{align*}
Therefore, \eqref{lambda1g} follows.
\end{proof}

\begin{proof}[Proof of Theorem \ref{Theorem::Cheeger}]
Within the proof $V^+, V^-$ and $g$ are as in the proof of Lemma \ref{Lemma:lambda1g}.

Without loss of generality, we can assume that $\weight (V^+)\preceq \weight (V^-)$ (otherwise, consider eigenfunction $-v_1$). Moreover, we can enumerate all $n=\#V$ vertices of the graph so, that 
\begin{equation*}
\underbrace{v_1(x_1)\preceq\dots \preceq v_1(x_{n_0})\preceq 0}_{V^-}\prec\underbrace{ v_1(x_{n_0+1})\preceq\dots\preceq v_1(x_{n})}_{V^+}.
\end{equation*}
We introduce some further notations. Let
\begin{equation*}
S_i = \{x_j\;:\; 1\le j \le i\},\; i= {1,\dots, n}
\end{equation*}
and
\begin{equation*}
\beta = \min_{i= {1,\dots, n}} \dfrac{\weight(\partial S_i)}{\min\{\weight(S_i), \weight(V\setminus S_i)\}}.
\end{equation*}
Then
\begin{equation*}
\weight(\partial S_i) = \sum_{\substack{j: 1\le j \le i\\k: i<k\le n}}  \weight (x_j,x_k)
\end{equation*}
\begin{equation*}
\weight(S_i) = \sum_{j: 1\le j \le  i}  \weight (x_j)\mbox{ and } \weight(V\setminus S_i) = \sum_{j: i< j \le n}  \weight (x_j)
\end{equation*}
and we note that $h\preceq\beta$.

Denoting the r.h.s of \eqref{lambda1g} by $W$ and multiplying numerator and denominator by $\sum_{x,y} \weight (x,y)(g(x)+g(y))^2$ we obtain
\begin{equation*}
W = \dfrac{\frac{1}{2}\sum_{x,y}\weight (x,y)(g(x)-g(y))^2 \sum_{x,y}\weight (x,y) (g(x)+g(y))^2}{\sum_x \weight (x)g(x)^2 \sum_{x,y} \weight (x,y)(g(x)+g(y))^2}
\end{equation*}
and by Cauchy-Schwarz inequality \footnote{The Cauchy-Schwarz inequality in an ordered field can be proven by induction or through Pythagorean theorem} 
\begin{equation}\label{W1}
W\succeq \dfrac{\frac{1}{2}\left(\sum_{x,y}\weight (x,y)\left|g^2(x)-g^2(y)\right|\right)^2}{\sum_x \weight (x)g(x)^2 \sum_{x,y} \weight (x,y)(g(x)+g(y))^2}
\end{equation}
Apart from this from \eqref{lambda1g} we have
\begin{align*}
W\sum_x \weight (x)g(x)^2&= \dfrac{1}{2}\sum_{x,y} \weight (x,y) (g(x)-g(y))^2 =[\mbox{ Green formula}]\\
&=\sum_{x}\L g(x)g(x)\weight (x)=\sum_{x} \weight (x)g^2(x)- \sum_{x,y} \weight (x,y)g(y)g(x)\\
&=[\pm\mbox{the first sum}]\\
&=2\sum_{x} \weight (x)g^2(x)-\sum_{x} \weight (x)g^2(x) - \sum_{x,y} \weight (x,y) g(y)g(x)\\
&=2\sum_{x} \weight (x)g^2(x)-\sum_{x,y} \weight (x,y)g^2(x) -\sum_{x,y} \weight (x,y) g(y)g(x)\\
&=2\sum_{x} \weight (x)g^2(x)-\left(\sum_{x,y} \weight (x,y)g(x)(g(x)+g(y))\right)\\
&=2\sum_{x} \weight (x)g^2(x)-\left(\sum_{x,y} \weight (x,y)g(y)(g(x)+g(y))\right)\\
&=2\sum_{x} \weight (x)g^2(x)-\dfrac{1}{2}\left(\sum_{x,y} \weight (x,y)(g(x)+g(y))^2\right),\\
\end{align*}
where in the last two lines we exchanged notations $x$ and $y$, sum the two lines up and divided by $2$ (usually this trick is used in a proofs of Green's formula).

Therefore, from \eqref{W1} follows
\begin{align*}
W&\succeq \dfrac{\frac{1}{2}\left(\sum_{x,y}\weight (x,y)|g^2(x)-g^2(y)|\right)^2}{\sum_x \weight (x)g(x)^2 (4\sum_{x} \weight (x)g^2(x)-2W\sum_x \weight (x)g(x)^2)}\\
&\succeq \dfrac{\frac{1}{2}\left(\sum_{x,y}\weight (x,y)|g^2(x)-g^2(y)|\right)^2}{\left(\sum_x \weight (x)g(x)^2\right)^2 (4-2W)}
\end{align*}
Now let us estimate the numerator $A:=\frac{1}{2}\left(\sum_{x,y}\weight (x,y)|g^2(x)-g^2(y)|\right)^2$. 
\begin{align*}
A&=\dfrac{1}{2} \left(2\sum_{i<j} \weight (x_i,x_j)(g^2(x_j)-g^2(x_i))\right)^2=2 \left(\sum_{i=1}^{n-1}\sum_{j=i+1}^n \weight (x_i,x_j)(g^2(x_j)-g^2(x_i))\right)^2\\
&=2 \left(\sum_{i=1}^{n-1}\sum_{j=i+1}^n \weight (x_i,x_j)(g^2(x_j)-g^2(x_{j-1})+g^2(x_{j-1})\dots-g^2(x_{i+1})+g^2(x_{i+1})-g^2(x_i))\right)^2\\
&=2 \left(\sum_{i=1}^{n-1}\sum_{j=i+1}^n \weight (x_i,x_j)\sum_{k=i}^{j-1}(g^2(x_{k+1})-g^2(x_k))\right)^2\\
&=2 \left(\sum_{i=1}^{n-1}\sum_{j=i+1}^n\sum_{k=i}^{j-1} \weight (x_i,x_j)(g^2(x_{k+1})-g^2(x_k))\right)^2\\
&=[\mbox {change limits of summation for $k$ and $j$}]\\
&=2 \left(\sum_{i=1}^{n-1}\sum_{k=i}^{n-1}\sum_{j=k+1}^{n} \weight (x_i,x_j)(g^2(x_{k+1})-g^2(x_k))\right)^2\\
&=2 \left(\sum_{i=1}^{n-1}\sum_{k=i}^{n-1}(\weight (x_i,x_{k+1})+\dots+\weight (x_i,x_n))(g^2(x_{k+1})-g^2(x_k))\right)^2\\
&=[\mbox {change limits of summation for $k$ and $i$}]\\
&=2 \left(\sum_{k=1}^{n-1}\sum_{i=1}^{k}(\weight (x_i,x_{k+1})+\dots+\weight (x_i,x_n))(g^2(x_{k+1})-g^2(x_k))\right)^2\\
&=2 \left(\sum_{k=1}^{n-1}\underbrace{\sum_{i=1}^{k}\sum_{j=k+1}^{n}\weight (x_i,x_j)}_{=\weight(\partial S_k)}(g^2(x_{k+1})-g^2(x_k))\right)^2\\
&=2 \left(\sum_{k=1}^{n-1}\weight(\partial S_k)(g^2(x_{k+1})-g^2(x_k))\right)^2=2 \left(\sum_{k=n_0}^{n-1}\weight(\partial S_k)(g^2(x_{k+1})-g^2(x_k))\right)^2.\\
\end{align*}
Note that for $k\ge n_0$ we have $\weight(V\setminus S_k)\preceq \weight( S_k)$ due to $ V^-\subset S_k$. Therefore 
$$
\dfrac{\weight(\partial S_k)}{\weight(V\setminus S_k)}\succeq \beta
$$
for such $k$ and we obtain
\begin{align*}
A&\succeq 2 \left(\sum_{k=n_0}^{n-1}\beta \weight(V\setminus S_k)(g^2(x_{k+1})-g^2(x_k))\right)^2\\
&= 2 \left(\beta\sum_{k=n_0}^{n-1}\sum_{j=k+1}^n \weight (x_j)(g^2(x_{k+1})-g^2(x_k))\right)^2\\
&= 2 \left(\beta\sum_{j=n_0+1}^{n}\sum_{k=n_0}^{j-1} \weight (x_j)(g^2(x_{k+1})-g^2(x_k))\right)^2\\
&= 2 \left(\beta\sum_{j=n_0+1}^{n}\weight (x_j)(g^2(x_{j})-g^2(x_{n_0}))\right)^2= 2 \left(\beta\sum_{j=n_0+1}^{n}\weight (x_j)g^2(x_{j})\right)^2,\\
&= 2 \left(\beta\sum_{j=1}^{n}\weight (x_j)g^2(x_{j})\right)^2,\\
\end{align*}
where the last two equalities are true, since $g(x_i)=0$ for $i\le n_0$ and $V^-\ne \emptyset$.

Therefore,
\begin{align*}
W&\succeq \dfrac{A}{\left(\sum_x \weight (x)g(x)^2\right)^2 (4-2W)}\succeq \dfrac{2\beta^2}{4-2W}\\
\end{align*}
From where follows ($W\preceq 2$ as Rayleigh quotient of $g$), that
$$
W^2-2W+\beta^2\preceq 0
$$
and 
$W\succeq 1-\sqrt{1-\beta^2}$ since discriminant is positive due to $\beta\preceq 1=\dfrac{\weight(\partial S_1)}{\weight(S_1)}$.

\end{proof}
Using Lemma \ref{lemma::non-complete} we immediately get the estimate of $\alpha_1$ in terms of Cheeger constant:
\begin{corollary}\label{cor::Cheegeralpha}
For any  non-complete graph the following holds:
$$\alpha_1=|\alpha_1|\preceq \sqrt{1-h^2}.$$
\end{corollary}

\subsection{An estimation of $\alpha_{n-1}$}

\begin{lemma}\label{aaa} For any graph with at least two vertices  the following holds:
$\alpha_{n-1}\prec 0$ and $|\alpha_{n-1}|\succeq \frac{1}{n-1}=c_n\in \Bbb Q^+\subset \K$.
\end{lemma}
\begin{proof}
It follows from \eqref{lambdan-1:n:n-1} since $\alpha_{n-1}=1-\lambda_{n-1}$.
\end{proof}
\section{Properties of $\P^m f$ and its subsequences}
Through this section let $\P$ be a transition operator on finite graph $(V,b)$ over non-Archimedean field $\K$ and $f\in \mathfrak F$. We investigate a behaviour of the sequence $\{\P^m f\}_{m\in \Bbb N}$ as $m\to \infty$. Firstly we present estimates on $\{\P^m f\}_{m\in \Bbb N}$ in terms of $f$. The results are similar to the case of classical weighted graphs. 
\subsection{Estimates for $\P^m f$}
\label{section::rw}
\begin{theorem}\label{Thm4}
For any non-complete graph with $\#V=n>2$ and for any $f\in \lspan(\{v_i \mid \alpha_i\succ 0, i=\overline{0,n-1}\})$ the following holds
 \begin{equation*}
\|\P^{m} f - \widetilde f\|\preceq \alpha_1^m \|f\|,
\end{equation*}
where
\begin{equation*}
\widetilde f(x)=\dfrac{1}{\weight (V)}\sum_{y\in V} f(y)\weight (y),\; x\in V.
\end{equation*}
\end{theorem}

\begin{proof}
Let $f=\sum\limits_{i=0}^{k}f_i v_i$, where $\displaystyle k=\max_{\substack{i=\overline{0,n-1};\\\alpha_i\succ 0}} i$. Then 
\begin{equation*}
\P^{m} f-\widetilde f=\sum\limits_{i=0}^{k}f_i\alpha_i^m v_i-\widetilde f
=f_0v_0+\sum\limits_{i=1}^{k}f_i\alpha_i^{m} v_i-\widetilde f,
\end{equation*}
since $\alpha_0=1$. Further, 
using that $v_0$ is a constant normalized eigenfunction, i.e. $v_0\equiv \frac{1}{\sqrt{b(V)}}$, we get
\begin{align*}
f_0&=\langle f,v_0\rangle=\langle f, \dfrac{1}{\sqrt{\weight (V)}}\rangle=\dfrac{1}{\sqrt{\weight (V)}} \sum_{y\in V} f(y)\weight (y).\\
\end{align*}
Therefore,
\begin{align*}
f_0v_0&=\dfrac{1}{{\weight (V)}} \sum_{y\in V} f(y)\weight (y)=\widetilde f
\end{align*}
and 
\begin{align*}
\P^{m} f-\widetilde f&=\sum\limits_{i=1}^{k}f_i\alpha_i^{k} v_i.
\end{align*}
Then 
\begin{align*}
\|\P^{m} f-\widetilde f\|^2&=\langle \P^{m} f-\widetilde f, \P^{2m} f-\widetilde f\rangle=\sum\limits_{i=1}^{k}f_i^2\alpha_i^{2m}\\
&\preceq \max_{1\le i\le k}\alpha_i^{2m}\sum\limits_{i=1}^{k}f_i^2\preceq \max_{1\le i\le k} \alpha_i^{2m}\sum\limits_{i=0}^{k}f_i^2=\max_{1\le i\le k} \alpha_i^{2m}\|f\|^2.
\end{align*}
and we obtain
\begin{align*}
\|{\P^m} f-\widetilde f\|&\preceq \alpha_1^{m}\|f\|,
\end{align*}
since eigenvalues are ordered by their values. The condition of non-completeness is needed to ensure $\alpha_1\succ 0$.
\end{proof}
Applying now Corollary \ref{cor::Cheegeralpha} we obtain:
\begin{corollary}
Under the conditions of Theorem \ref{Thm4} the following holds
$$
\|\P^{m} f - \widetilde f\|\preceq \left(\sqrt{1-h^2}\right)^m \|f\|.
$$
\end{corollary}

\begin{theorem}\label{Thm::bip}
For bipartite graph with the partition $V=V_1\cup V_2$ the following holds for any $f\in \mathcal F$:
 \begin{equation*}
\|\P^{2m} f - \overline f\|\preceq \alpha_1 ^{2m}\|f\|,
\end{equation*}
where
\begin{equation*}
\overline f(x)=\dfrac{2}{\weight (V)}
\begin{cases}
\sum_{y\in V_1} f(y)\weight (y),\; x\in V_1,\\
\sum_{y\in V_2} f(y)\weight (y),\; x\in V_2.
\end{cases}
\end{equation*}
\end{theorem}
The proof follows from explanations to Theorem 2.7 in \cite{Grigoryan} and follows the same outline as the proof of Theorem \ref{Thm4}.
\begin{proof}
Let $f=\sum_{i=0}^{n-1}f_i v_i$. Then 
\begin{align*}
\P^{2m} f-\overline f&=\sum\limits_{i=0}^{n-1}f_i\alpha_i^{2m} v_i-\overline f\\
&=f_0v_0+f_{n-1}v_{n-1}+\sum\limits_{i=1}^{n-2}f_i\alpha_i^{2m} v_i-\overline f,
\end{align*}
since $\alpha_0=1$ and for bipartite graphs $\alpha_{n-1}=-1$.
Let us notice that 
\begin{align*}
f_0&=\langle f,v_0\rangle=\langle f, \dfrac{1}{\sqrt{\weight (V)}}\rangle=\dfrac{1}{\sqrt{\weight (V)}} \sum_{y\in V} f(y)\weight (y)\\
&=\dfrac{1}{\sqrt{\weight (V)}} \sum_{y\in V_1} f(y)\weight (y)+\dfrac{1}{\sqrt{\weight (V)}} \sum_{y\in V_2} f(y)\weight (y)
\end{align*}
and
\begin{align*}
f_{n-1}&=\langle f,v_{n-1}\rangle=\dfrac{1}{\sqrt{\weight (V)}} \sum_{y\in V_1} f(y)\weight (y)-\dfrac{1}{\sqrt{\weight (V)}} \sum_{y\in V_2} f(y)\weight (y)
\end{align*}
Therefore,
\begin{align*}
f_0v_0&=\dfrac{1}{{\weight (V)}} \sum_{y\in V_1} f(y)\weight (y)+\dfrac{1}{{\weight (V)}} \sum_{y\in V_2} f(y)\weight (y)
\end{align*}
and
\begin{equation*}
f_{n-1}v_{n-1}=
\begin{cases*}
\dfrac{1}{{\weight (V)}} \sum_{y\in V_1} f(y)\weight (y)-\dfrac{1}{{\weight (V)}} \sum_{y\in V_2} f(y)\weight (y),\; x \in V_1\\
-\dfrac{1}{{\weight (V)}} \sum_{y\in V_1} f(y)\weight (y)+\dfrac{1}{{\weight (V)}} \sum_{y\in V_2} f(y)\weight (y),\; x \in V_2
\end{cases*}
\end{equation*}
from where follows $f_0v_0+f_{n-1}v_{n-1}-\overline f=0$ and 
\begin{align*}
\P^{2m} f-\overline f&=\sum\limits_{i=1}^{n-2}f_i\alpha_i^{2m} v_i.
\end{align*}
Then 
\begin{align*}
\|\P^{2m} f-\overline f\|^2&=\langle \P^{2m} f-\overline f, \P^{2m} f-\overline f\rangle=\sum\limits_{i=1}^{n-2}f_i^2\alpha_i^{4m}\\
&\preceq \max_{1\le i\le n-2} \alpha_i^{4m}\sum\limits_{i=1}^{n-2}f_i^2\preceq \max_{1\le i\le n-2} \alpha_i^{4m}\sum\limits_{i=0}^{n-1}f_i^2=\max_{1\le i\le n-2} \alpha_i^{4m}\|f\|^2.
\end{align*}
and we obtain
\begin{align*}
\|\P^{2m} f-\overline f\|^2&\preceq \alpha_1^{2m}\|f\|^2,
\end{align*}
since $\alpha_1=-\alpha_{n-2}$ and the values of other eigenvalues are between them.
\end{proof}

\begin{corollary}\label{Cor3}
Let $\#V>2$. Under the conditions of Theorem \ref{Thm::bip}  the following holds
$$
\|\P^{2m} f - \overline f\|\preceq \left({1-h^2}\right)^{m} \|f\|
$$
\end{corollary}
\begin{proof}
$\#V>2$ provides non-completeness of bipartite graph and we can use Corollary \ref{cor::Cheegeralpha}.
\end{proof}

\subsection{Property of non-convergence for non-bipartite graphs}
In classical case the estimates from Theorem \ref{Thm4} and Theorem \ref{Thm::bip} are mainly used to prove the convergence of random walk to equilibrium and to estimate mixing times (see \cite{Grigoryan}, \cite{LPW}). Here we investigate the convergence over non-Archimedean field, since any Archimedean field is isomorphic to a subfield of real numbers and therefore is not interesting for us.

The convergence of sequence $\{k_n\}_{n\in \Bbb N}$ to $k$ in ordered field means by definition, that for any element $a\in \K$ there exists $N_0\in \Bbb N$ such that for any $n\ge N_0$ we have
$$
|k_n-k|\preceq a.
$$
To state any result on convergence for an arbitrary ordered field is impossible, since in surreal numbers (which is the largest possible ordered field, containing all the others, see e.g. \cite{Conway}, \cite[Theorem 24.29]{BB}) only constant sequences converge, if one consider the above mentioned classical notion of convergence (\cite[Theorem 16]{RSS}). Therefore, we will present the theorem on convergence for the important case of the Levi-Civita field in the next section.

However, for non-bipartite graphs we have the following surprisingly interesting property of $\P^m$ over arbitrary non-Archimedean ordered field $\Bbb K$.
\begin{theorem}\label{Prop3}
Let $V$ be a non-bipartite graph over a non-Archimedean field $\K$. Then there is no limit $\mathcal \alpha_{n-1}^m v_{n-1}(x)$ for $m\to \infty$ for any vertex $x\in V$ with $v_{n-1}(x)\ne 0$ over any non-Archimedean ordered field.
There are no partial limits.
\end{theorem}
In other words, there always exists a function $f\in \mathfrak F$ and a vertex $x\in V$ such that $\P^{m}f(x)$ has no (partial) limit in $\Bbb K$.
\begin{proof}
For any two elements of the sequence $\alpha_{n-1}^{m} v_{n-1}(x)$ and $\alpha_{n-1}^{m+l} v_{n-1}(x)$, where $l\in \Bbb N$
we have
\begin{multline*}
|\alpha_{n-1}^{m} v_{n-1}(x)-\alpha_{n-1}^{m+l} v_{n-1}(x)|=|\alpha_{n-1}^{m}||1-\alpha_{n-1}^l|| v_{n-1}(x)|\\
\succ {\tau}|1+\alpha_{n-1}|{|v_{n-1}(x)|},
\end{multline*}
where $\tau\in\K$ is any infinitesimal, since $|\alpha_{n-1}^{m}|\succeq 1/(n-1)^m\succ \tau$ due to Lemma \ref{aaa} and for $|1-\alpha_{n-1}^l|$ we have two possibilities:
\begin{itemize}
\item
if $l$  is odd, then $|1-\alpha_{n-1}^l|\succ 1\succ |1+\alpha_{n-1}|$, since $-1\prec \alpha_{n-1}\prec 0$ for non-bipartite graphs.
\item
If $l$  is even, then $\alpha_{n-1}^l=|\alpha_{n-1}|^l\prec |\alpha_{n-1}|\prec 1$, since graph is non-bipartite, and
$|1-\alpha_{n-1}^l|\succ |1-|\alpha_{n-1}||=|1+\alpha_{n-1}|$.
\end{itemize}
Since $-1\prec \alpha_{n-1}$ for non-bipartite graph we have $|1+\alpha_{n-1}|\ne 0$ and
non-convergence follows from the fact that any subsequence is not a Cauchy sequence.

\end{proof}

\section{Mixing over the Levi-Civita field}
\label{section::LC}
\subsection{Preliminaries on the Levi-Civita field}
The Levi-Civita field was  introduced by Tullio Levi-Civita in 1862 \cite{LeviCivita}.
The Levi-Civita field is the smallest non-Archimedean real-closed ordered field, which is Cauchy complete in the order topology. Moreover, it has a subfield, isomorphic to the field of rational functions, whose elements naturally appear as edge weights in the theory of electrical networks (see e.g. \cite{Muranova1, Muranova2}).

Firstly, we recall the definition and main notations for the Levi-Civita field $\mathcal R$ (see e.g. \cite{Berz}, \cite{Hall}, \cite{Shamseddine}, \cite{Shamseddinethesis}, \cite{ShamseddineBerz}).

\begin{definition}
A subset $Q$ of the rational numbers $\Bbb Q$ is called \emph{left-finite} if for every number $r\in \Bbb Q$ there are only finitely many elements of $Q$ that are smaller than $r$. 
\end{definition}

\begin{definition}\cite{LeviCivita}\label{LCfield}
We define the Levi-Civita field $\mathcal R$ as a field of formal power series on $\epsilon$, where any non-zero element  $a\in \mathcal R$ can be written as
\begin{equation}
a=\sum_{i=0}^\infty a_i\epsilon^{q_i},
\end{equation}
with $a_i\in \Bbb R$, $a_0\ne 0$, $Q=\{q_i\}_{i=0}^\infty\subset \mathbb Q$ being left-finite and $q_i$ are pairwise different.
\end{definition}
Addition and multiplication are defined naturally as for formal power series. The order is defined as follows: $a\succ 0$ if $a_0>0$.

For 
\begin{equation}
a=\sum_{i=0}^\infty a_i\epsilon^{q_i}\mbox{ and }d=\sum_{i=0}^\infty d_i\epsilon^{p_i}\
\end{equation}
we write $a\approx d$ (``$a$ is {\em comparable} to $d$'') if $q_0=p_0$ and $a_0=d_0$.
For the element $a=0$ we put $q_0=\infty$. Therefore, two non-zero elements are comparable, if the first terms of their series coincide.  Zero is comparable only to itself.
%

For the convenience, we state the following lemma, which follows immediately from the notion of comparabless.

\begin{lemma}\label{smalllemma}
The following two statements hold:
\begin{enumerate}
 \item
Let  $a\approx d$ for $a,d\in \R\setminus\{0\}$, then $|a|\succ\epsilon |d|$. Moreover, from $a\approx d$ follows $|a|\approx |d|$ and $a^m\approx d^m$ for any $m\in \Bbb N$.
\item
Let $a\approx r$ for $a\in \R, r\in \mathbb R\setminus\{0\}$. Then $|a|\succ \epsilon$ and there exists $ q>0$ such that $|a-r|\prec \epsilon^q$.
\item
If $a_1\approx d_1$ and $a_2\approx d_2$, where $a_1,a_2,d_1,d_2\in \R$, then $a_1a_2\approx d_1d_2$.
\end{enumerate}
\end{lemma}
\subsection{Properties of $\P^m f$ in the Levi-Civita field} Let us consider finite graphs over the Levi-Civita field. Two following Propositions describe completely the behaviour of eigenfunctions of $\P$ under the action of $\P^m, m\in \Bbb N$ depending on eigenvalues.
\begin{proposition}\label{PropLC}
If $\alpha_i$ is an eigenvalue of transition operator such that $\alpha_i\approx r$ for some $r\in \Bbb R\setminus\{0\}$ and $\alpha_i\ne 1$ then there is no limit $\mathcal \alpha_{i}^m v_{i}(x)$ for $m\to~\infty, m\in\Bbb N$ whenever $v_{i}(x)\ne 0$.
Moreover, if in addition $\alpha_i\ne -1$, then there are no partial limits.
\end{proposition}
\begin{proof}
Firstly the case $\alpha_i=~-1$ is obvious, since partial limits for $\alpha_i=~-1$ exist for the subsequences $\mathcal \alpha_{i}^{2m} v_{i}(x)$ and $\mathcal \alpha_{i}^{2m+1} v_{i}(x)$ and limit does not exists whenever $v_i(x)\ne 0$. Therefore, let further in the proof $a_i\ne \pm 1$.

Let us point out that $|r|\le 1$ follows from  $|\alpha_{i}|\preceq 1$ .

For any two elements of sequence $\alpha_{i}^{m} v_{i}(x)$ and $\alpha_{i}^{m+l} v_{i}(x)$, where $l\in \Bbb N$
we have
\begin{equation*}
|\alpha_{i}^{m} v_{i}(x)-\alpha_{i}^{m+l} v_{i}(x)|=|\alpha_{i}^{m}||1-\alpha_{i}^l|| v_{i}(x)|.
\end{equation*}
For the first multiplayer $|\alpha_i^{m}|\approx |r^m|$ holds and $|\alpha_i^{m}|\succ \epsilon$ due to Lemma \ref{smalllemma}. Further we have several cases:
\begin{enumerate}
\item
for $r^l\ne 1$, we have  $|1-\alpha_{i}^l|\approx |1- r^l|\in \Bbb R\setminus\{0\}$ and we obtain 
\begin{equation}\label{epskw}
|\alpha_{i}^{m} v_{i}(x)-\alpha_{i}^{m+l} v_{i}(x)|\succ{\epsilon}^2| v_{i}(x)|.
\end{equation}
due to Lemma \ref{smalllemma} (2).
\item
for $r^l=1$ (i. e. $r=1$  or $r=-1$ with $l$ being even) we have 
$$
|1-~\alpha_{i}^l|\approx~|\alpha_{i,1}lr^{l-1}\epsilon^{q_1}| \mbox{ for }
\alpha_{i}=r+\sum_{j=1}^\infty\alpha_{i,j}\epsilon ^{q_j}, \mbox{ with } \alpha_{i,1}\ne 0
$$
and
$|1-~\alpha_{i}^l|\succ\epsilon^{q_1+1} |\alpha_{i,1}lr^{l-1}|$ due to Lemma \ref{smalllemma} (1),. Therefore,
$$
|\alpha_{i}^{m} v_{i}(x)-\alpha_{i}^{m+l} v_{i}(x)|\succ{\epsilon}^{2+q_1}|\alpha_{i,1}lr^{l-1}|| v_{i}(x)|
$$ 
and again since $\alpha_{i,1}\ne 0$ we obtain
\begin{equation}\label{eps3}
|\alpha_{i}^{m} v_{i}(x)-\alpha_{i}^{m+l} v_{i}(x)|\succ{\epsilon}^{3+q_1}| v_{i}(x)|.
\end{equation}
\end{enumerate}
From \eqref{epskw} and \eqref{eps3} we obtain that for any $m, l\in\Bbb N$ holds
$$
|\alpha_{i}^{m} v_{i}(x)-\alpha_{i}^{m+l} v_{i}(x)|\succ{\epsilon}^{3+q_1}| v_{i}(x)|,
$$ 
i.e. the sequence $\alpha_{i}^{m} v_{i}(x)$ has no Cauchy subsequences.
\end{proof}

\begin{proposition}\label{PropLC1}
If $\alpha_i$ is an eigenvalue of transition operator such that $\alpha_i\approx ~r\epsilon^q$, for some $r\in \Bbb R, q>0,  q\in\Bbb Q\cup\{\infty\}$ then there is a limit $\mathcal \alpha_{i}^m v_{i}(x)$ for $m\to ~\infty, m\in \Bbb N$ for any vertex $x\in V$.
\end{proposition}
\begin{proof}
Note that we also include here the case $\alpha_i=0$ ($\Leftrightarrow \alpha_i\approx 0$), since in this case $\alpha_{i}^m v_{i}(x)= 0\to 0$ for $m\in \Bbb N$.

Therefore, further in the proof we assume that $r\ne 0$. We have 
$$|\alpha_{i}^{m} v_{i}(x)|\approx |r^m\epsilon^{qm} v_i(x)|.
$$
 Let $\displaystyle v_i(x)=\sum_{j=0}^\infty v_{i,j} \epsilon^{t_j}$. Then
$$
|\alpha_{i}^{m} v_{i}(x)|=|r^m||v_{i,0}(x)|\epsilon^{qm+t_0} +o(\epsilon^{qm+t_0}),
$$
where  $o(\epsilon^p)$ stays for terms with powers of $\epsilon$, higher then $p$. Therefore, for any $a\succ 0, a=\sum_{i=0}^\infty a_i\epsilon^{q_i}\in \R$
we obtain ${|\alpha_{i}^{m} v_{i}(x)|\prec a}$  taking $m$ large enough (such that $qm+t_0>q_0$).
\end{proof}

Now we present convergence results in terms of Cheeger constant $h$ defined by \eqref{Def:Cheeger}.

\begin{theorem}\label{corhLC}
Let $V$ be a  bipartite graph with $\#V>2$ over the Levi-Civita field.
Let $h\approx 1$. Then for any $f\in \mathfrak F$ we have $\P^{2m} f\to \overline f$, where 
$$\overline f=\dfrac{2}{b(V)}
\begin{cases}
\sum_{y\in V_1} f(y)\weight (y),\; x\in V_1,\\
\sum_{y\in V_2} f(y)\weight (y),\; x\in V_2.
\end{cases}
$$
\end{theorem}
\begin{proof}

Since $\#V>2$, the bipartite graph is non-complete and $\alpha_1\succeq 0$. Since $h\preceq 1$, from $h^2\approx 1$ we get $1-h^2=|1-h^2|\prec \epsilon^q$ for some $q>0$ due to Lemma \ref{smalllemma} (2) and the convergence follows from Corollary \ref{Cor3}, since 
$$
\|\P^{2m}-\overline f\|\prec \epsilon^{qm} \|f\|\prec  \epsilon^{qm+t_0}\varphi_0 \mbox{ for } \|f\|=\sum_{i=0}^{\infty}\varphi_i \epsilon^{t_i}.
$$
Therefore, for any $a\succ 0, a=\sum_{i=0}^\infty a_i\epsilon^{q_i}\in \R$
we obtain ${\|\P^{2m}-\overline f\|\prec a}$  taking $m$ large enough (such that $qm+t_0>q_0$).
\end{proof}

\begin{theorem}\label{corhLC1}
Let $V$ be a non-complete non-bipartite graph with $\#V=n>1$ over the Levi-Civita field. Let $h\approx 1$. Then for any $f\in \lspan(\{v_i \mid \alpha_i\succ 0, i=\overline{0,n-1}\})$ we have $\P^{m} f\to \widetilde f$, where 
\begin{equation*}
\widetilde f(x)=\dfrac{1}{\weight (V)}\sum_{y\in V} f(y)\weight (y),\; x\in V.
\end{equation*}
\end{theorem}
\begin{proof}
Follows from Theorem \ref{Thm4} by Corollary \ref{cor::Cheegeralpha}, similarly to the proof of Theorem \ref{corhLC}.
\end{proof}

\subsection{Examples}

\begin{example}
Let us consider a graph over the Levi-Civita field with $4$ vertices, presented at Figure \ref{Fig1} with $R,D>0$ (notations $R$ and $D$ come from electrical networks and make sense there, see e.g. \cite{Muranova1}, \cite{Muranova2}).

\begin{figure}[H]\label{Fig1}
\centering
\begin{tikzpicture}[auto,node distance=2.5cm,
                    thick,main node/.style={circle, draw, fill=black!100,
                        inner sep=0pt, minimum width=3pt}]

  \node[main node] (1) [label={[above]$1$}]{};
  \node[main node] (2) [right of=1,label={[above]$2$}] {};
  \node[main node] (3) [right of=2,label={[above]$3$}] {};
  \node[main node] (4) [right of=3,label={[above]$4$}] {};

  \path[every node/.style={font=\sffamily\small}]
    (2) edge node [bend left] {$R^{-1}$} (1)
    (3) edge node [bend right] {$R^{-1}$} (2)
    (4) edge node [bend right] {$D^{-1}\epsilon$} (3);

\end{tikzpicture}
\caption{}
\end{figure}
The matrix of the operator $\mathcal P$ is the following:
\begin{equation*}
\left(
\begin{array}{cccc}
0&1&0&0\\
\frac{1}{2}&0&\frac{1}{2}&0\\
0&\frac{D}{D+R\epsilon}&0&\frac{R\epsilon}{D+R\epsilon}\\
0&0&1&0\\
\end{array}\right)
\end{equation*}
The corresponding eigenvalues are $\alpha_{0,3}=\pm 1$ and $\alpha_{1,2}=\pm \sqrt{\frac{R\epsilon}{2D+2R\epsilon}}$ (note, that the graph is bipartite). The latest eigenvalues can be written as generalized binomial series (which gives as a homomorphism from rational functions on $\epsilon$ to $\R$) as
\begin{align*}
\alpha_{1,2}&=\pm \sum_{i=0}^\infty {{-\frac{1}{2}} \choose_i}\frac{1}{\sqrt{2}}\left(\frac{R}{D}\right)^{i+\frac{1}{2}}\epsilon^{i+\frac{1}{2}}\\
&=\pm \frac{1}{\sqrt{2}}\left(\frac{R}{D}\right)^\frac{1}{2}\epsilon^{\frac{1}{2}}\mp\frac{1}{2\sqrt{2}}\left(\frac{R}{D}\right)^{\frac{3}{2}}\epsilon^{\frac{3}{2}}+o(\epsilon^{\frac{3}{2}}),
\end{align*}
where ${{-\frac{1}{2}} \choose_k}$ is a generalized binomial coefficient.

Therefore, $|\alpha_{1,2}|\prec \left(\frac{R\epsilon}{2 D}\right)^\frac{1}{2}$ and
 \begin{equation*}
\|\P^{2m} f - \overline f\|\prec \left(\frac{R\epsilon}{2 D}\right)^{m}\|f\|,
\end{equation*}
from where follows, that $\P^{2m} f \to \overline f$ as $m$ goes to $\infty$.

We can also easily calculate Cheeger constant for a given graph (considering all possible partitions) and get 
$$
h=\frac{D}{D+2R\epsilon}
$$ 
that holds on the partition $\{1,2\}\cup\{3,4\}$.

Further, 
$$
h=\sum_{i=0}^\infty (-1)^i\left(\frac{2R}{D}\right)^{i}\epsilon^{i}\approx 1
$$
and we can use Theorem \ref{corhLC} to state the convergence of $\P^{2m} f $.
\end{example}

\begin{example}
Let us investigate the eigenvalues of the graph  over the Levi-Civita field, presented at Figure \ref{Fig2} with $R,D>0$.
\begin{figure}[H]\label{Fig2}
\centering
\begin{tikzpicture}[auto,node distance=2.5cm,
                    thick,main node/.style={circle, draw, fill=black!100,
                        inner sep=0pt, minimum width=3pt}]

  \node[main node] (1) [label={[above]$1$}]{};
  \node[main node] (2) [right of=1,label={[above]$2$}] {};
  \node[main node] (3) [right of=2,label={[above]$3$}] {};
  \node[main node] (4) [right of=3,label={[above]$4$}] {};

  \path[every node/.style={font=\sffamily\small}]
    (2) edge node [bend left] {$R^{-1}$} (1)
    (3) edge node [bend right] {$D^{-1}\epsilon$} (2)
    (4) edge node [bend right] {$R^{-1}$} (3);

\end{tikzpicture}
\caption{}
\end{figure}
The corresponding eigenvalues are $\alpha_{0,3}=\pm 1$ and $\alpha_{1,2}=\pm {\frac{D}{D+R\epsilon}}$. The latest eigenvalues can be written in form generalized  binomial series as 
$$
\alpha_{1,2}=\pm \sum_{i=0}^\infty {{-1} \choose_i}\left(\dfrac{R}{D}\right)^i \epsilon^{i}=\pm 1\mp\frac{R\epsilon}{D}+o(\epsilon).
$$

Therefore, $|\alpha_{1,2}|\approx 1$.


The Cheeger constant for a given graph is 
$$
h=\frac{R\epsilon}{2D+R\epsilon}=\frac{R\epsilon}{2D}-\left(\frac{R\epsilon}{2D}\right)^2+o(\epsilon^2).
$$
This example has shown us, that a small (smaller than any natural number) Cheeger constant is possible over $\R$, as well as eigenvalues of $\P$ comparable to $1$, and the convergence of $\P^m f$ fails for general $f$.
\end{example}

\end{document}